\documentclass[a4paper]{amsart} 
\usepackage{graphicx}
\usepackage{amsmath,amssymb,textcomp,graphicx,amsthm} 
\usepackage{epsfig}

\usepackage{color}
\usepackage[alphabetic]{amsrefs}
 \usepackage{verbatim}

\newcommand{\e}{\varepsilon}
\newcommand{\inflap}{\lap_\infty}

\newtheorem*{thm}{Theorem}

\newcommand{\dd}{\partial}

\newcommand{\dist}{\operatorname{dist}}
\newcommand{\lap}{\Delta}

\renewcommand{\div}{\operatorname{div}}

\title{The $\infty$-harmonic potential is not always an $\infty$-eigenfunction}
\author{Erik Lindgren}

\begin{document} 
\begin{abstract}\noindent
In this note we prove that there is a convex domain for which the $\infty$-harmonic potential is not a first $\infty$-eigenfunction.
\end{abstract}
\maketitle

\noindent This note is devoted to a counter example. For a bounded domain $\Omega$, a non-negative function $u\in C(\overline\Omega)$, vanishing on $\dd\Omega$, is said to be a \emph{first $\infty$-eigenfunction} if 
\begin{equation}\label{eq:eig}
\max \left(\Lambda  -\frac{|\nabla u(x)|}{u(x)},\inflap u (x)\right)=0\quad \textup{in $\Omega$,}
\end{equation}
where
$$
\inflap u =\sum_{i,j=1}^n\frac{\dd u}{\dd x_i}\frac{\dd u}{\dd x_j}\frac{\dd^2 u}{\dd x_i \dd x_j},
$$
is the infinity Laplace operator and 
$$
\Lambda = \left(\sup \{r: B_r(x)\subset \Omega \text{ for some $x\in \Omega$}\}\right)^{-1}
$$
Equation \eqref{eq:eig} has to be understood in the viscosity sense. This $\infty$-eigenvalue problem arises as the limit of the eigenvalue problem for the $p$-Laplace operator,
$$
\div (|\nabla u|^{p-2}\nabla u)+\lambda_p |u|^{p-2}u=0,
$$ and was first studied in \cite{JLM99}.

The \emph{distance function}, $\delta (x)=\dist(x,\dd\Omega)$ and the set where it attains its maximum value, the \emph{High ridge} $\Gamma = \{x:\delta (x)=\max_x\delta(x)\}$, play central roles. In general, a first $\infty$-eigenfunction is not unique up to a multiplicative constant (see \cite{HSY12} for a counter example). However, for certain geometries where $\delta$ is the first $\infty$-eigenfunction this is indeed the case, see \cite{Yu07}. This includes for instance the ball and the stadium  (the convex hull of two balls with the same radii), but not the square. In \cite{JLM99}, it is shown that $\delta$ is not a first $\infty$-eigenfunction in the square.
It has been suggested that the so-called $\infty$-harmonic potential, the solution $v$ of
\begin{equation}\label{eq:infpot}
\left\{ \begin{array}{lr}\inflap v = 0 \text{ in $\Omega\setminus \Gamma$,}\\
 v = 0 \text{ on $\dd\Omega$,}\\
 v = 1 \text{ on $\Gamma$}.
 \end{array}\right. 
\end{equation}
is a first $\infty$-eigenfunction. Indeed, this is true when $\Omega$ is a ball or a stadium (cf. \cite{Yu07}), but it is not known whether this is true or not for the square. Below we show that there is a convex domain for which the $\infty$-harmonic potential is not a first $\infty$-eigenfunction.
\begin{figure}[!ht]
\begin{center}
     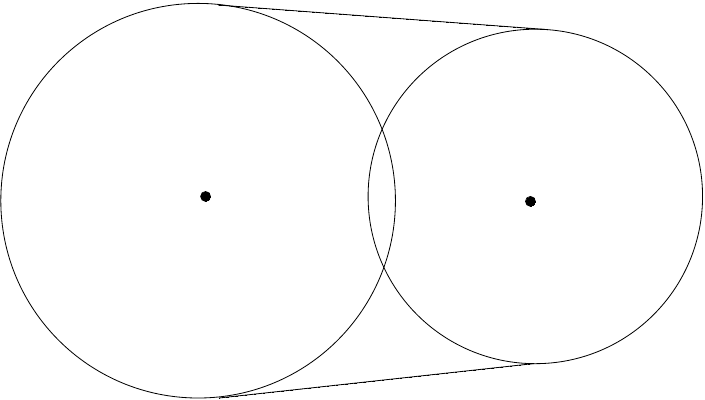
 \caption{$S_\e$}
   \label{fig:S}  
\end{center}
\end{figure}

For $\e\geq 0$ let 
$$S_\e=B_{1+\e}(-1,0)\cup B_1(1,0)\cup R,$$
where $B_r(x)$ denotes the ball of radius $r$ centered at $x$ and $R$ the region between the two lines joining the points $(-1,\pm (1+\e))$ and $(1,\pm 1)$ (see Figure \ref{fig:S}). In particular, we see that $S_0$ is a stadium, i.e., the convex hull of the set $B_1(1,0)\cup B_1(-1,0)$. We note that in $S_\e$, the High ridge consists only of the point $(-1,0)$ for any $\e>0$, while for $\e=0$, it is $\{(t,0):t\in [-1,1]\}$. This is the key observation.

\begin{thm} For any $\e$ small enough, the $\infty$-harmonic potential $v_\e$ in $S_\e$ (as defined in \eqref{eq:infpot}), is not a first $\infty$-eigenfunction.
\end{thm}
\begin{proof} Suppose towards a contradiction that there is a sequence $\e_j\to 0$, such that $v_j=v_{\e_j}$ is a first $\infty$-eigenfunction. Then, since the $v_j$'s are all first eigenfunctions, they have a uniformly bounded Lipschitz norm, globally in $\overline S_{\e_j}$ (any first $\infty$-eigenfunction minimizes the $L^\infty$-Rayleigh quotient). Hence, we can, by the Arzela-Ascoli theorem, extract a subsequence, again labelled  $v_j$, converging uniformly to a limit function $v_0$ in $\overline S_0$. By standard arguments for viscosity solutions it follows that $v_0$ is a first $\infty$-eigenfunction in $S_0$. From the uniqueness theorem in \cite{Yu07}, we know that $v_0=\delta$ and thus $v_0(1,0)=1$.

In addition, since $\inflap v_j=0$ in $S_\e\setminus (-1,0)$, it follows that $\inflap v_0=0$ in $S_0\setminus (-1,0)$. Then the strong maximum principle for $\infty$-harmonic functions (see for instance \cite{LM95}) implies $v_0(1,0)<1$, which contradicts $v_0=\delta$.
\end{proof}
\bibliographystyle{amsrefs}
\bibliography{countref}

\begin{bibdiv}
\begin{biblist}

\bib{HSY12}{article}{
      author={Hynd, Ryan},
      author={Smart, Charlie~K.},
      author={Yu, Yifeng},
       title={Nonuniqueness of infinity ground states},
        date={2012},
     journal={preprint},
}

\bib{JLM99}{article}{
      author={Juutinen, Petri},
      author={Lindqvist, Peter},
      author={Manfredi, Juan~J.},
       title={The {$\infty$}-eigenvalue problem},
        date={1999},
     journal={Arch. Ration. Mech. Anal.},
      volume={148},
      number={2},
       pages={89\ndash 105},
}

\bib{LM95}{article}{
      author={Lindqvist, Peter},
      author={Manfredi, Juan~J.},
       title={The {H}arnack inequality for {$\infty$}-harmonic functions},
        date={1995},
        ISSN={1072-6691},
     journal={Electron. J. Differential Equations},
       pages={No.\ 04, approx.\ 5 pp.\ (electronic)},
}

\bib{Yu07}{article}{
      author={Yu, Yifeng},
       title={Some properties of the ground states of the infinity
  {L}aplacian},
        date={2007},
     journal={Indiana Univ. Math. J.},
      volume={56},
      number={2},
       pages={947\ndash 964},
}

\end{biblist}
\end{bibdiv}
\end{document}